%% file: loops_in_taut.tex
\documentclass[12pt]{amsart}

\input{header_basic.tex}

\input{header_article.tex}
\input{header_subtle.tex}

\usepackage{mathtools}

\usepackage{enumitem}
\usepackage{color}
\usepackage{subfig, caption}
\usepackage{wrapfig}
\usepackage{pdflscape}
\usepackage{array}

\newcommand{\ind}{\operatorname{ind}}
\newcommand{\cusps}{\operatorname{cusps}}

\title[Essential loops]{Essential loops in taut ideal triangulations}     

\author[Schleimer]{Saul Schleimer}
\address{\hskip-\parindent
        Mathematics Institute\\
        University of Warwick\\
        Coventry CV4 7AL, United Kingdom}
\email{s.schleimer@warwick.ac.uk}

\author[Segerman]{Henry Segerman} 
\address{\hskip-\parindent
        Department of Mathematics\\
        Oklahoma State University\\
        Stillwater, OK, 74078, USA}
\email{segerman@math.okstate.edu}

\thanks{This work is in the public domain.}

\date{\today}

\begin{document}

\begin{abstract}
In this note we combinatorialise a technique of Novikov.  
We use this to prove that, in a three-manifold equipped with a taut ideal triangulation, 
any vertical or normal loop is non-trivial in the fundamental group. 
\end{abstract}

\maketitle



\section{Introduction}
\label{Sec:Intro}

The notion of a taut ideal triangulation of a three-manifold is due to Lackenby~\cite{Lackenby00}.  
He combinatorialised the angle structures introduced independently by Casson and by Rivin~\cite{Rivin94}.  
They in turn linearised the geometric triangulations of Thurston~\cite{Thurston78}.
Each of these structures plays an important role in modern low-dimensional topology.  In particular, taut ideal triangulations have a strong connection to the subject of taut foliations, introduced by Gabai~\cite{Gabai83}, 
and to that of taut branched surfaces, due to Oertel~\cite{Oertel86}.  
In addition to the results of Lackenby, taut ideal triangulations play a central role in the theory of layered triangulations.  One spectacular contribution has been as a prerequisite for Agol's theory of veering triangulations~\cite{Agol11}.  

Novikov~\cite[Theorem~6.1]{Novikov65} gives one of the early applications of foliations to the study of the fundamental group of a manifold.  He starts with a loop $\delta$ in good position with respect to a foliation $\calF$.  He further supposes that $H \from D \to M$ is a null-homotopy of $\delta$, also in good position.  Pulling back, he obtains a singular foliation $H^{-1}(\calF)$ on the disk $D$.  The Poincar\'e--Hopf theorem gives combinatorial control of the singularities, which translates to topological control over the homotopy.  Morally, the positivity of the Euler characteristic of the disk constrains the position of $\delta$.  We refer to Candel and Conlon~\cite[Chapter~9]{CandelConlon03} for a history of the subject and for detailed proofs. 

We introduce a combinatorial version of the Novikov technique; instead of pulling back a foliation we pull back a taut ideal triangulation.  This gives a train track with stops in the disk $D$.  We so obtain a very simple proof of a variant of one of Novikov's results.   That is, suppose that $M$ is a three-manifold, equipped with a taut ideal triangulation $\calT$.  Let $\calB = \calT^{(2)}$ be the resulting branched surface in $M$. 

\begin{restate}{Theorem}{Thm:VerticalLoopsEssential}
Any loop $\delta$ in $M$ which is vertical with respect to $\calB$ is non-trivial in $\pi_1(M)$. 
\end{restate}

There is also an indirect proof of this using Novikov's original technique~\cite[Theorem~4.35(3)]{Calegari07}, once we observe that $\calB$ carries an essential lamination which extends to a taut foliation of $M$ (see \cite[Example~5.1]{GabaiOertel89} as well as \cite[page~373]{Lackenby00}).  

Using our techniques we also obtain a new result, as follows.

\begin{restate}{Theorem}{Thm:NormalLoopsEssential}
Any loop $\gamma$ in $M$ which is normal with respect to $\calB$ is non-trivial in $\pi_1(M)$. 
\end{restate}

The proof of \refthm{NormalLoopsEssential} is more delicate than that of \refthm{VerticalLoopsEssential}; new behaviour near the boundary of $D$ must be dealt with.  

From Theorems~\ref{Thm:VerticalLoopsEssential} and~\ref{Thm:NormalLoopsEssential} we deduce that vertical, and also normal, loops are infinite order in the fundamental group.  Note that this is a bit weaker than the conclusion in the comparable situation of a train track $\tau$ in a surface -- there loops dual to, or carried by, $\tau$ are not only non-trivial but also non-peripheral.  

We have a simple corollary of \refthm{NormalLoopsEssential}.  
Let $\cover{M}$ be the universal cover of $M$ and let $\cover{\calB}$ be the resulting branched surface.

\begin{corollary}
Suppose that $F$ is a connected surface (perhaps with boundary) carried by $\cover{\calB}$ and realised as a (perhaps finite) union of faces of $\cover{\calB}$.  Then $F$ is a disk. \qed
\end{corollary}



\subsection*{Previous work} 
Gabai and Oertel prove that laminations carried by essential branched surfaces are $\pi_1$--injective~\cite[Lemma~2.7]{GabaiOertel89}.  Our \refthm{NormalLoopsEssential} is both more and less general than their work.  We do not require a lamination.  They do not require the manifold to be cusped. 

Calegari~\cite[Remark~5.6]{Calegari00} gives a very different combinatorial version of \refthm{VerticalLoopsEssential}, in the closed case.  He introduces the notion of a \emph{local orientation}; this is, in a sense, dual to having a transverse taut branched surface $\calB \subset M$ where all components of $M - \calB$ are taut balls. 


\subsection*{Acknowledgements} 
We thank Marc Lackenby for helpful conversations. 
The second author was supported in part by National Science Foundation grant DMS-1708239.

\section{Background}
Throughout the paper we will use $M$ to denote a compact connected manifold with non-empty boundary.  All boundary components will be tori or Klein bottles.  Suppose that $\calT$ is a three-dimensional triangulation; that is, a collection of \emph{model tetrahedra} and a collection of face pairings.  We will also call the faces of a model tetrahedron \emph{model faces}, and similarly for its edges and vertices.

Let $|\calT|$ be the quotient space: that is, we take the disjoint union of the model tetrahedra of $\calT$ and identify model faces using the face pairings.  
Let $\calT^{(k)}$ be the $k$--skeleton of $|\calT|$.  Let $n(\calT^{(0)})$ be an open regular neighbourhood of the vertices of $\calT$.   We call $\calT$ a \emph{ideal triangulation} of $M$ if $|\calT| - n(\calT^{(0)})$ is homeomorphic to $M$. 

A \emph{taut angle structure} on $\calT$ is an assignment of dihedral angles, zero or $\pi$, to each model edge in $\calT$.  The assignment is required to obey two conditions.  The \emph{edge equalities} state that, for an edge $e \in \calT^{(1)}$, the sum of the dihedral angles of its models is $2\pi$.  The \emph{triangle equalities} state that, for any model vertex, the sum of the dihedral angles of the three adjacent model edges is $\pi$.  We say that the tetrahedra of $\calT^{(3)}$ are \emph{taut}.  See \reffig{TautTet}.

\begin{figure}[htbp]
\centering
\subfloat[A taut tetrahedron.]{
\labellist
\small\hair 2pt
\pinlabel {$0$} at 20 130
\pinlabel {$0$} at 240 120
\pinlabel {$0$} at 135 27
\pinlabel {$0$} at 135 217
\pinlabel {$\pi$} at 125 140
\pinlabel {$\pi$} at 125 87
\endlabellist
\includegraphics[width=0.35\textwidth]{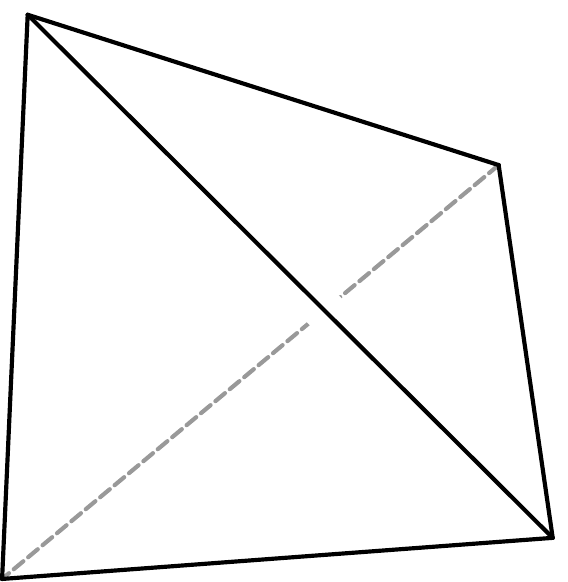}
\label{Fig:TautTet}
}
\quad 
\subfloat[All faces meeting a single edge in $\calB$.]{
\raisebox{0.85cm}{\includegraphics[width=0.55\textwidth]{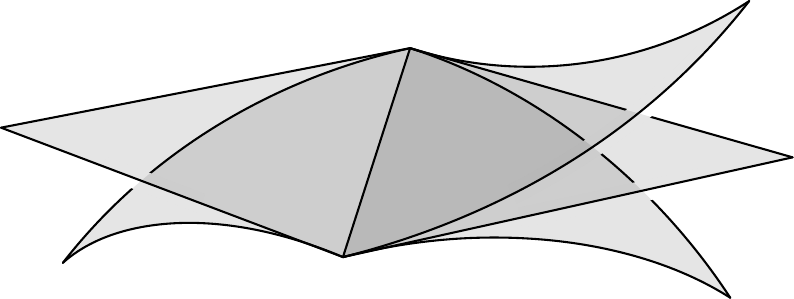}}
\label{Fig:TautEdge}
}
\caption{}
\label{Fig:Taut}
\end{figure}

We deduce that every taut tetrahedron has four edges with dihedral angle zero.  We call the union of these four edges the \emph{equator} of the taut tetrahedron.  

Suppose now that $e$ is an edge of $\calT^{(1)}$.  There are exactly two model edges for $e$ with angle $\pi$; all others are zero.  Obeying these dihedral angles, we isotope the two-skeleton $\calT^{(2)}$ to obtain a smooth \emph{branched surface} $\calB$.  See \reffig{TautEdge}.  Some references would call $\calB$ a non-generic branched surface without vertices.  See for example~\cite[Section~6.3]{Calegari07}.



\begin{definition}
Suppose that $\delta$ is a smooth embedded loop in $M$.
Suppose that $\delta$ is transverse to, and meets, $\calB$.  
Suppose that for every tetrahedron $t$ we have that every arc $d$ of $\delta \cap t$ \emph{links} the equator of $t$.  
(That is, the endpoints of $d$ are separated in $\bdy t$ by the equator of $t$.)  
Then we say that the loop $\delta$ is \emph{vertical} with respect to $\calB$. 
\end{definition}

\begin{definition}
\label{Def:NormalLoop}
Suppose that $\gamma$ is a smooth loop immersed in $\calB$. 
Suppose that $\gamma$ is transverse to, 
and meets, 
the edges of $\calB$.  
Suppose that for every model face $f$ of $\calB$ and for every component $J$ of $\gamma^{-1}(f)$, the arc $\gamma|J$ is \emph{normal} in $f$.  
(That is, the endpoints of $\gamma|J$ lie in distinct edges of $f$.)  
Then we say that the loop $\gamma$ is \emph{normal} with respect to $\calB$.
\end{definition}

\section{Combinatorics of null-homotopies}

Suppose that $\delta$ is a loop in $M$ which is transverse to the branched surface $\calB$.  Let $D = D^2$ be the unit disk with the usual orientation.  Suppose that $H \from D \to M$ is a null-homotopy of $\delta$.  We homotope $H$ relative to $\bdy D$ to make $H$ transverse to $\calB$.

We define $\tau = H^{-1}(\calB)$.  Thus $\tau$ is a \emph{train track} in $D$.   The \emph{switches} of $\tau$ are exactly the points of $H^{-1}(\calB^{(1)})$.  The \emph{stops} of $\tau$ are exactly the points of $(H|\bdy D)^{-1}(\calB)$.  
The standard reference for train tracks is \cite{PennerHarer92}; we also rely on \cite{Mosher03}.  We note that our track $\tau$ does not satisfy the so-called ``geometry-condition''~\cite[page~5]{PennerHarer92},~\cite[page~52]{Mosher03}.  

We call a connected component $R$ of $D - \tau$ a \emph{region}.  Let $\cusps(R)$ and $\corners(R)$ count the number of (necessarily outwards) cusps and corners on the boundary of $R$.  As a bit of terminology, we divide $\bdy R$ into \emph{sides}: these are the components of $\bdy R$ minus all outward cusps and corners.  Note that a side $s$ of $R$ may be a union of several branches of $\tau$.  

We define the \emph{index} of $R$ to be 
\[
\ind(R) = \chi(R) - \frac{\cusps(R)}{2} - \frac{\corners(R)}{4}
\] 
In \reftab{DiskIndices} we give pictures of, and names to, all possible disk regions with non-negative index.  Note that index is additive under taking the union of regions~\cite[page~57]{Mosher03}.
Thus the sum of the indices of the regions of $D - \tau$ is exactly $\chi(D)$; that is, one.  We deduce from this that there is at least one region $R$ with positive index. 

\def\w{0.17}
\def\m{0.16}
\setlength\tabcolsep{2pt}
\newcolumntype{V}[1]{>{\centering\arraybackslash} m{#1\textwidth} <{\vspace{-1pt}}}
\begin{table}[htbp]
\centering
\begin{tabular}{V{0.22}|V{\w}V{\w}V{\w}}
  $\corners \backslash \textrm{index}$ & $1$ & $1/2$ & $0$ \\
\hline \\ [-12pt]
    0	& \includegraphics[width=\m\textwidth]{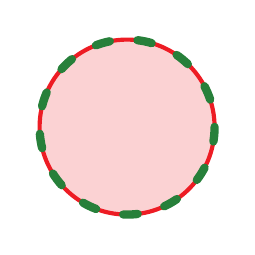} 
	& \includegraphics[width=\m\textwidth]{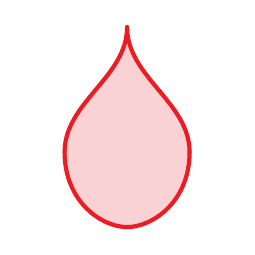} 
	& \includegraphics[width=\m\textwidth]{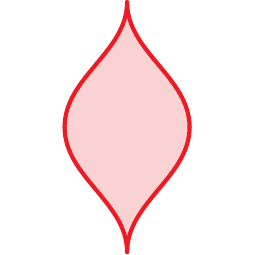} \\
    2	&  
	& \includegraphics[width=\m\textwidth]{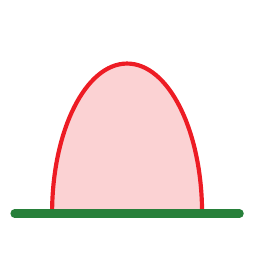} 
	& \includegraphics[width=\m\textwidth]{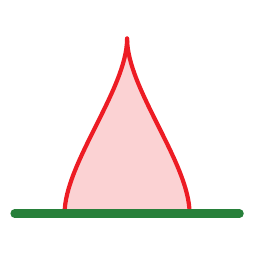} \\
   4	& 
	&
 	& \includegraphics[width=\m\textwidth]{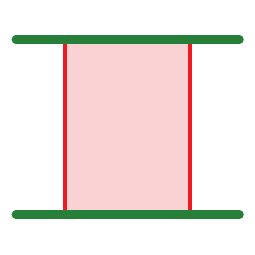} \\
\end{tabular}
\vspace{15pt}
\caption{Disk regions with non-negative index, organised by the number of their corners.  
These are named as follows: \emph{nullgon}, \emph{cusped monogon}, \emph{cusped bigon}, \emph{boundary bigon}, \emph{boundary trigon}, and \emph{rectangle}.
If $\tau$ is empty then $D$ is called a \emph{complete nullgon}.
}
\label{Tab:DiskIndices}
\end{table}

Let $r(H)$ be the number of regions of $D - \tau$.  Over all null-homotopies of $\delta$, transverse to $\calB$, we choose $H$ to minimise $r(H)$.   We call such an $H$ \emph{minimal}. 

\begin{lemma}
\label{Lem:Minimality}
Suppose that $\delta$ is a loop in $M$ transverse to $\calB$.  
Suppose that $H \from D \to M$ is a minimal null-homotopy of $\delta$.  
Let $\tau = H^{-1}(\calB)$.  
Suppose that $R$ is a region of $D - \tau$. 
Then we have the following. 
\begin{enumerate}
\item
\label{Itm:Disk}
$R$ is a disk. 
\item
\label{Itm:NormalPlus}
Suppose that $s$ is a side of $R$.  
Then the interior of $s$ meets at most one switch. 
\item
\label{Itm:Nullgon}
If $R$ is a nullgon, then $R = D$ is a complete nullgon.
\item
\label{Itm:Monogon}
$R$ is not a cusped monogon. 
\end{enumerate}
Thus, the positive index regions of $D$ are either a single complete nullgon, or a 
non-empty collection of boundary bigons.
\end{lemma}

\begin{proof}
\mbox{}
\begin{enumerate}
\item
If $R$ has topology, then we can compress it into the containing tetrahedron and reduce $r(H)$. 
\item
Suppose that the interior of $s$ meets at least two switches. 
All such switches in the interior of $s$ are preimages under $H$ of a single edge. 
Hence there is a branch $b \subset \tau$ so that $H(b)$ is a non-normal arc.  
We homotope $H$ in a neighbourhood of $b$ to make $H(b)$ simple. 
This done, $H(b)$ cuts a bigon $B$ off of the face containing $H(b)$. 
We then homotope $H$ across $B$. 
This does not increase $r(H)$.  
If $r(H)$ does not decrease, then this move disconnects $\tau$, and creates a region with topology, contradicting \refitm{Disk}.
\item
Suppose that $R$ is a nullgon, not equal to $D$.  
If $H(\bdy R)$ is disjoint from $\calB^{(1)}$ then the region adjacent to $R$ is not a disk, contradicting \refitm{Disk}.  
It follows that $\bdy R$ consists of an even number of branches of $\tau$ (alternating between the two faces of a tetrahedron $t$ on either side of a $\pi$-edge of $t$).  
But this contradicts \refitm{NormalPlus}.
\item
Suppose that $t$ is the taut tetrahedron containing $H(R)$.  
Let $s$ be the boundary of $R$.  
We deduce that the loop $s$ crosses the equator of $t$ exactly once, a contradiction.  
\end{enumerate}
Thus, the only possible positive index regions are complete nullgons and boundary bigons. 
\end{proof}

Equipped with this we can now prove the following. 

\begin{theorem}
\label{Thm:VerticalLoopsEssential}
Let $(M, \calT)$ be a three-manifold equipped with a taut ideal triangulation.  Let $\calB = \calT^{(2)}$ be the resulting branched surface in $M$.  Any loop $\delta$ in $M$ which is vertical with respect to $\calB$ is non-trivial in $\pi_1(M)$. 
\end{theorem}

\begin{proof}
Suppose that $H \from D \to M$ is a minimal null-homotopy of the vertical loop $\delta$.  Applying \reflem{Minimality}, there must be a region $R$ of $D - \tau$ which is a boundary bigon.  Let $t$ be the tetrahedron containing $H(R)$.  Let $d = \bdy R \cap \bdy D$ and let $s = \bdy R - d^\circ$.  From the definition of vertical, we have that $H(d)$ links the equator of $t$.  Therefore $H(s)$ crosses the equator of $t$ an odd number of times, and thus at least once.  This contradicts the fact that $\bdy R$ has no cusps. 
\end{proof}

\section{Transverse taut}

In order to prove \refthm{NormalLoopsEssential}, we will use the following strengthening of the notion of a taut structure. A \emph{transverse taut structure} on $\calT$ is a taut structure together with a co-orientation on $\calB$ with the following property.  If model faces $f$ and $f'$ of a model tetrahedron $t$ share a common model edge $e$, then 
\begin{itemize}
\item the edge $e$ is part of the equator of $t$ if and only if exactly one of the co-orientations on $f$ and $f'$ points into $t$.
\end{itemize}
\noindent
See \reffig{TransverseTet}.  It follows that the co-orientations on faces incident to an edge change direction precisely twice as we go around an edge.  See \reffig{TransverseEdge}. 

\begin{figure}[htbp]
\centering
\subfloat[Co-orientations and angles in a transverse taut tetrahedron.]{
\labellist
\small\hair 2pt
\pinlabel {$0$} at 20 130
\pinlabel {$0$} at 240 120
\pinlabel {$0$} at 135 27
\pinlabel {$0$} at 135 217
\pinlabel {$\pi$} at 125 140
\pinlabel {$\pi$} at 125 87
\endlabellist
\includegraphics[width=0.35\textwidth]{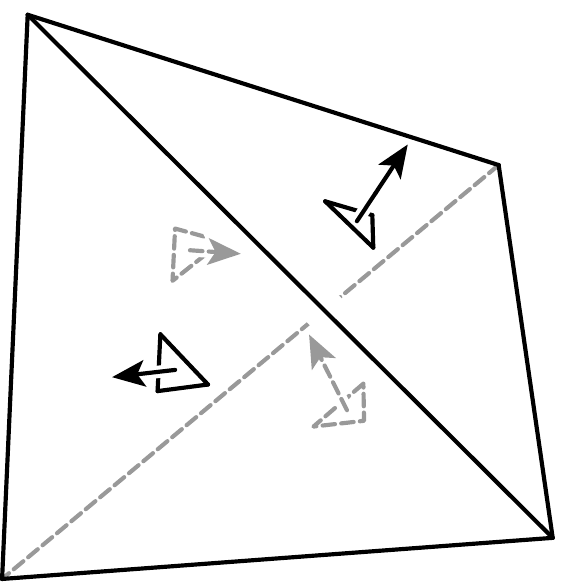}
\label{Fig:TransverseTet}
}
\quad
\subfloat[Co-orientations around an edge.]{
\raisebox{0.85cm}{\includegraphics[width=0.55\textwidth]{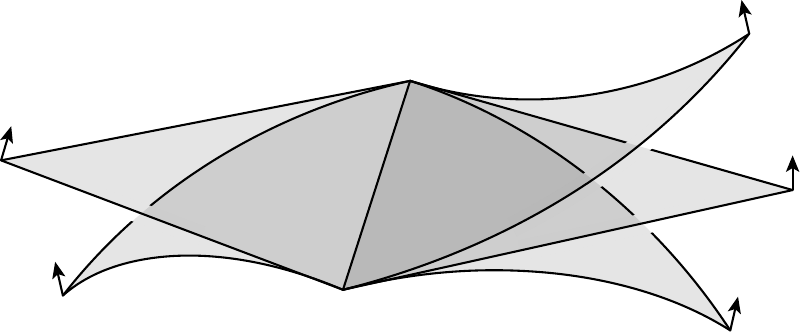}}
\label{Fig:TransverseEdge}
}
\caption{}
\label{Fig:Transverse}
\end{figure}

Suppose that $\calT$ is an ideal triangulation of a manifold $M$ equipped with a taut structure.  We now construct a triangulation $\cover{\calT}$ of a double cover $\cover{M}$ of $M$.  By construction, the lift of the taut structure on $\calT$ to $\cover{\calT}$ will support a transverse taut structure. 

For each taut tetrahedron $t$ of $\calT$, we arbitrarily label the two model edges with dihedral angle $\pi$ as $e'$ and $e''$. In $\cover{\calT}$, we have two taut tetrahedra $t'$ and $t''$ corresponding to $t$. We assign a co-orientation to the model faces of $t'$ and $t''$ in such a way that the co-orientation points into the tetrahedron on the two model faces of $t^*$ incident to $e^*$. Now suppose that $t_i$ and $t_j$ are tetrahedra of $\calT$, glued to each other along model faces $f_i$ and $f_j$.  In $\cover\calT$ we have tetrahedra $t'_i$, $t''_i$, $t'_j$, and $t''_j$, with model faces $f'_i$, $f''_i$, $f'_j$, and $f''_j$, respectively.
 
We glue $t'_i$ to either $t'_j$ or $t''_j$ as the co-orientation on $f'_i$ agrees with $f'_j$ or $f''_j$.  We similarly glue $t''_i$ to the remaining copy of $t_j$.  Having made all such gluings, the resulting triangulation $\cover{\calT}$ has a transverse taut structure by construction.  It has one component if and only if the taut structure on $\calT$ does not support a transverse taut structure. 

\section{Proof of the main result}

\begin{theorem}
\label{Thm:NormalLoopsEssential}
Let $(M, \calT)$ be a three-manifold equipped with a taut ideal triangulation.  Let $\calB = \calT^{(2)}$ be the resulting branched surface in $M$.  Any loop $\gamma$ in $M$ which is normal with respect to $\calB$ is non-trivial in $\pi_1(M)$. 
\end{theorem}

\begin{proof}
Suppose for a contradiction, that the normal loop $\gamma$ is null-homotopic.  Thus $\gamma$ lifts to a normal loop in any cover.  Thus, without loss of generality, we may assume that the taut structure on $\calT$ supports a transverse taut structure.  This gives us a local notion of \emph{upwards}.  In particular,  every model tetrahedron has two lower faces and two upper faces, separated by its equator.

\reflem{Minimality} does not apply directly to a normal loop $\gamma$.  So, let $A$ be a model annulus with horizontal boundary circles $\bdy_0 A \sqcup \bdy_1 A$.  Let $G$ be a small smooth homotopy $G \from A \to M$, moving $\gamma$ slightly upwards.  That is, $G(\bdy_0 A) = \gamma$ and we define $\delta = G(\bdy_1 A)$.  We ensure that $G$ is transverse to $\calB$ away from $\bdy_0 A$; also, we arrange that for each vertical interval $J$ in $A$ the tangents to $G(J)$ point upwards.  We will apply \reflem{Minimality} to $\delta$. 

We call $\delta$ a \emph{raised curve}.  We call the components of $\delta - \calB$ \emph{raised arcs}.  There are six \emph{types} of raised arc.  These are shown in \reffig{RaisedArcs}. 
There is a cellulation of $A$ with one-skeleton $\bdy A \cup G^{-1}(\calB)$. Suppose that $C$ is a two-cell. Let $c = C\cap \bdy_0 A$ and $d = C\cap \bdy_1 A$. Thus $G(c) \subset \gamma$ and $G(d)\subset \delta$. We say that $G(c)$ is the \emph{lowering} of the raised arc $G(d)$. We record this by the \emph{lowering map}, $L$, where $L(G(d)) = G(c)$. Note that $G(c)$ may be either a single vertex, a single normal arc, or two normal arcs.  Again, see \reffig{RaisedArcs}.

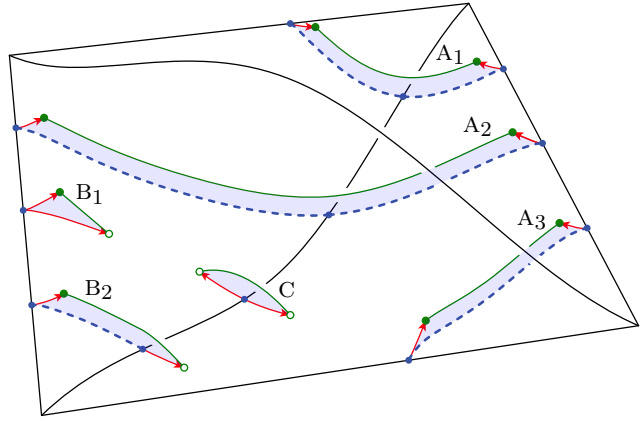
\begin{figure}[htbp]
\labellist
\small\hair 2pt
\pinlabel {$\textsc{a}_3$} at 405 152
\pinlabel {$\textsc{a}_2$} at 363 224
\pinlabel {$\textsc{a}_1$} at 342 280
\pinlabel {$\textsc{b}_1$} at 63 172
\pinlabel {$\textsc{b}_2$} at 70 98
\pinlabel {$\textsc{c}$} at 215 98
\endlabellist
\includegraphics[height = 5.5 cm]{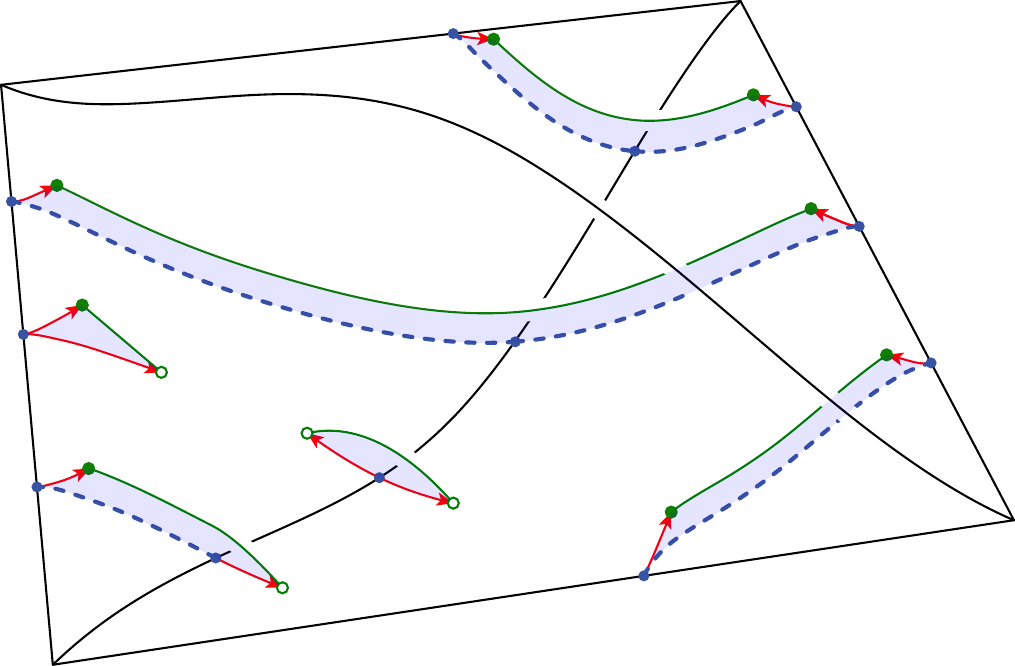}
\caption{A taut tetrahedron containing the six possible types of raised arcs of $\delta$.  These are drawn in (solid) green. The normal arcs (or points) of $\gamma$, namely the lowerings of the raised arcs, are drawn in (dashed) blue.  Images of the two-cells of the annulus $A$ are shaded in light blue.  Filled green dots indicate endpoints of raised arcs on the top two faces of the tetrahedron; open green dots indicate endpoints on the bottom two faces.}
\label{Fig:RaisedArcs}
\end{figure}


Suppose that $H \from D \to M$ is a minimal null-homotopy of $\delta$. 
Recall that $\tau = H^{-1}(\calB)$.
Applying \reflem{Minimality}, we find that $D - \tau$ has at least two boundary bigons.  
Applying another small homotopy, we can retain minimality and also make $H$ transverse to $\gamma$. 


Pulling back the transverse taut structure on $\calB$ by $H$ gives a transverse orientation on the branches of $\tau$ which is consistent across switches.  Thus, for any region $R$ of $D - \tau$ and for any side $s$ of $R$, the transverse orientation on $s$ points either into, or out of, $R$.  This gives us a classification of boundary bigons.  Suppose that $R$ is a boundary bigon and $s = \bdy R - \bdy D$ is its side in $\tau$.  If the transverse orientation on $s$ points out of $R$ then we call $R$ a \emph{min-bigon}.  If it points into $R$ we call $R$ a \emph{max-bigon}.  

\subsection{Min-bigons}
\label{Sec:MinBigon}

Suppose that $R$ is a min-bigon.  We move $\gamma$ up, across $H(R)$, to obtain $\gamma'$.  We appeal to \reflem{Minimality}\refitm{NormalPlus} to ensure that $\gamma'$ is normal.  Let $\delta'$ be the corresponding raised loop and let $H'$ be the new null homotopy.  See \reffig{PushOverMinBigon}. 

The loop $\gamma'$ may be shorter than, the same length as, or longer than $\gamma$ (see types $\textsc{a}_1, \textsc{a}_2$ and $\textsc{a}_3$ in \reffig{RaisedArcs}).  However, $H'$ has exactly one fewer region.  That is, $r(H') = r(H) - 1$.    We repeat this process until there are no more min-bigons.

\begin{figure}[htbp]
\includegraphics[width=\textwidth]{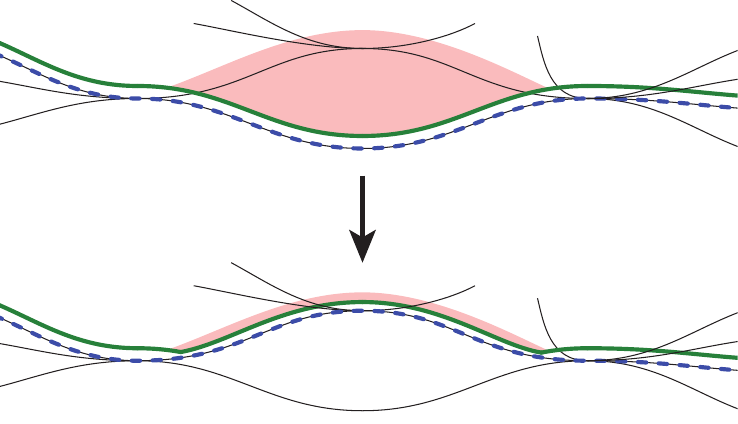}
\caption{Pushing over a min-bigon of type $A_3$.}
\label{Fig:PushOverMinBigon}
\end{figure}

\subsection{Max-bigons}
\label{Sec:MaxBigon}
Suppose that $R_0$ is a max-bigon. Unlike the situation of a min-bigon, a max-bigon does not give us a simple move to reduce complexity. The asymmetry stems from the fact that we raised $\gamma$ rather than lowered it. Instead, our plan is to uniquely associate to $R_0$ two small subregions of $D - \tau$, each with index $-1/4$.  This will imply that the index of $D$ is at most zero.  This contradiction finally proves \refthm{NormalLoopsEssential}.

We begin as follows.  Let $s$ be the side of $R_0$ in $\tau$.  Let $d_0 = \bdy R_0 - s \subset \bdy D$. We give $d_0$ the (tangential) orientation it receives from $D$.  In \reffig{DownhillFromMaxBigon}, this orientation will point left.  Note that $H(d_0)\subset \delta$ is a raised arc.  Let $c_0 = L(H(d_0))$ be its lowering.  

\newpage
\begin{claim}
\label{Clm:Base}
\mbox{}
\begin{itemize}
\item
The raised arc $H(d_0)$ has type $\textsc{c}$.  
\item
The side $s$ meets exactly one switch $c'_0$ of $\tau$.  
\item 
The vertices $c_0$ and $H(c'_0)$ cobound a sub-edge $\epsilon_0 \subset \calB^{(1)}$. 
\end{itemize}
\end{claim}

\begin{proof}
Let $t_0$ be the tetrahedron containing $H(R_0)$.  By the definition of a max-bigon the transverse orientation on $s$ points into $R_0$.  Thus each corner of $H(R_0)$ is contained in a lower face of $t_0$.  Consulting \reffig{RaisedArcs} we deduce that $H(d_0)$ is of type $\textsc{c}$.  Thus each corner of $H(R_0)$ is contained in its own lower face of $t_0$.  We deduce that $s$ meets at least one switch of $\tau$.
By \reflem{Minimality}\refitm{NormalPlus} the side $s$ meets exactly one switch, which we call $c'_0$.  

Since $H$ is transverse to $\gamma$, the vertices $c_0$ and $H(c'_0)$ are distinct.  They are contained in the same edge of $\calB^{(1)}$: namely the bottom edge $e_0$ of $t_0$. In $e_0$ they cobound a sub-edge, which we call $\epsilon_0$. 
\end{proof}

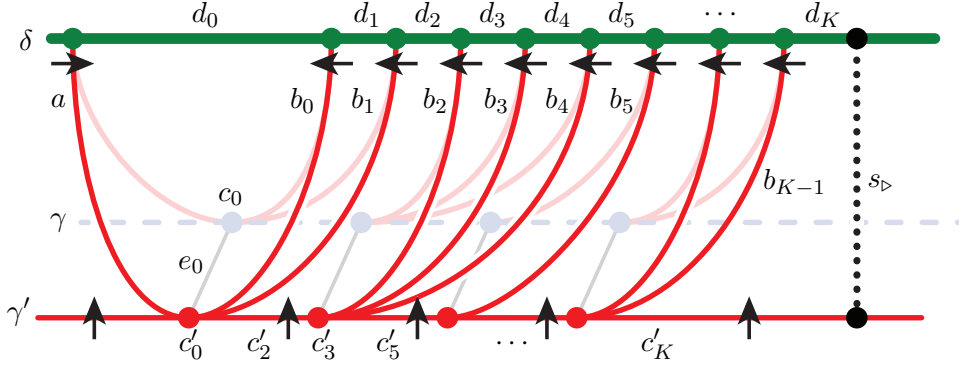
\begin{figure}[htbp]
\labellist
\small\hair 2pt
\pinlabel $\delta$ at 0 85
\pinlabel $\gamma$ at 9 34
\pinlabel $a$ at 9 68
\pinlabel $d_0$ at 50 92
\pinlabel $d_1$ at 95 92
\pinlabel $d_2$ at 112 92
\pinlabel $d_3$ at 130 92
\pinlabel $d_4$ at 148 92
\pinlabel $d_5$ at 166 92
\pinlabel $\cdots$ at 194 92
\pinlabel $d_K$ at 222 92
\pinlabel $b_0$ at 77 68
\pinlabel $b_1$ at 94 68
\pinlabel $b_2$ at 114 68
\pinlabel $b_3$ at 131 68
\pinlabel $b_4$ at 148 68
\pinlabel $b_5$ at 166 68
\pinlabel $b_{K-1}$ at 214 45
\pinlabel $c_0$ at 57 41
\pinlabel $s_\triangleright$ at 238 45
\pinlabel $\epsilon_0$ at 46 22
\pinlabel $c'_0$ at 46 0
\pinlabel $c'_2$ at 65 0
\pinlabel $c'_3$ at 83 0
\pinlabel $c'_5$ at 101 0
\pinlabel $\cdots$ at 136 0
\pinlabel $c'_K$ at 176 0
\endlabellist
\includegraphics[width = \textwidth]{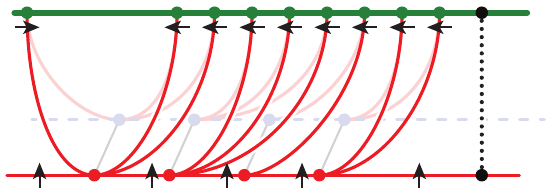}
\caption{A possible picture of part of the annulus $A$ (in back), the bigons $B_k$ (in front), and the homotopies $F_k$ (bottom).  To lighten the notation in this figure, we have omitted applying $H$ to labels of subsets of $D$. 
Transverse orientations on the branches $b_i$ are shown with arrows. Note that $\epsilon_0=\epsilon_1$, $c'_0=c'_1$, $c'_3=c'_4$, and so on.}
\label{Fig:DownhillFromMaxBigon}
\end{figure}

Let $a$ and $b_0$ be the components of $s - c'_0$, where $b_0$ meets the right endpoint of $d_0$. See the leftmost region on \reffig{DownhillFromMaxBigon}. 

Now consider a sequence of regions $R_0, R_1, \ldots, R_n$ that meet $\bdy D$ in the sides $d_0, d_1, \ldots d_n$ as we move along $\bdy D$ to the right. Define $c_k = L(H(d_k)) \subset \gamma$: the lowering of the raised arc $H(d_k)$.  Define $\gamma_k = \cup_{i = 0}^k c_i \subset \gamma$.

Let $b_i$ be the branch of $\tau$ that meets $\bdy D$ at the right corner of $R_i$.  We now choose $n=N$ such that $b_{i-1}$ and $b_i$ have the same transverse orientation for $1 \leq i < N$, while they have opposite transverse orientations for $i=N$. Thus, $R_N$ is at a \emph{local minimum} of $\gamma$, and we are going downhill to it from the \emph{local maximum} at $R_0$.  This downhill condition implies that for $i \in [1,N)$, the raised arc $H(d_i)$ is of type either $\textsc{b}_1$ or $\textsc{b}_2$.  Again, see \reffig{RaisedArcs}.  

Recall that all positive index regions are now max-bigons. Thus none of the $R_i$ can have positive index for $i > 0$.
Let $K$ be the smallest number for which $R_K$ has negative index, or if there is none, then set $K = N$.  

\begin{claim}
\label{Clm:NotTrigon}
The region $R_K$ is not a boundary trigon. 
\end{claim}
\begin{proof}
If $K<N$ then by definition, $R_K$ has negative index, and so is not a boundary trigon. If $K=N$ then $R_N$ cannot be a boundary trigon since the transverse orientations on the two sides of a boundary trigon must agree, yet $R_N$ is at a local minimum of $\gamma$.
\end{proof}

For all $k \in [0, K)$ we define the union $B_k = \cup_{i = 0}^k R_i$.  Define $\gamma'_k = \bdy B_k - (a^\circ \cup b^\circ_k \cup \bdy D)$.  (Unlike in \refsec{MinBigon}, here $\gamma'_k$ is a push-off of only a section of $\gamma$.)

\begin{definition}
\label{Def:TransverseHomotopy}
Suppose that $g, h \from [0,1] \to \calB$ are paths. Suppose that $F\from [0,1] \times [0,1] \to \calB$ is a homotopy from $g$ to $h$. Thus $g(x) = F(x,0)$ and $h(x) = F(x,1)$.
We say that $F$ is \emph{transverse} if whenever $F(x_0,t_0)$ is contained in a (1- or 2-) cell $C$ of $\calB$, we have that the \emph{trace} $F(x_0,[0,1])$ lies in $C$.
\end{definition}


\begin{claim}
\label{Clm:Induct}
For all $k \in [1,K)$:
\begin{enumerate}
\item 
\label{Itm:Trigon}
The region $R_k$ is a boundary trigon.  
\item 
\label{Itm:Bigon}
The union $B_k$ has exactly two corners and no cusps.  
\item 
\label{Itm:Homotopy}
There is a transverse homotopy $F_k$ taking $\gamma_k$ to $H(\gamma'_k)$. 
\end{enumerate}
\end{claim}

\begin{proof}
We will prove this by induction.  \refclm{Base} implies the base case (for $k = 1$) in a manner essentially identical to the general inductive step, so we omit its proof. 

Suppose that the hypotheses hold at step $k$.  Recall that $H(d_k)$ has type $\textsc{b}_1$ or $\textsc{b}_2$, so it has precisely one lower endpoint. Let $f_k$ be the face that contains the lower endpoint.  Let $p$ be the endpoint of $\gamma_k$, and let $e_k$ be the edge of $f_k$ containing $p$.  Let $\beta$ be the normal arc of $\gamma$ immediately after $p$.  Let $f_\beta$ be the face containing $\beta$. Viewed in a small neighbourhood of $e_k$, 
the faces $f_\beta$ and $f_k$ are on the same side (say the right side) of $e_k$, and $f_\beta$ is below $f_k$. 

Let $p'$ be the endpoint of $\gamma'_k$ meeting $b_k$.  By hypothesis \refitm{Homotopy}, the transverse homotopy $F_k$ takes $p$ to $H(p')$, with trace lying in $e_k$. Since $H$ is transverse to $e_k$ at $H(p')$, we deduce that $H(D)$ meets both $f_k$ and $f_\beta$ at $H(p')$. Thus $p'$ is a switch of $\tau$ with a cusp immediately below $b_k$, to the right of $p'$, pointing at $\gamma'_k$ (which extends to the left of $p'$).  This cusp lies in $R_{k+1}$, since $b_k$ is part of the boundary of $R_{k+1}$.  See \reffig{DownhillFromMaxBigon}. 

If $R_{k+1}$ has negative index then $k+1 = K$ and we have nothing to prove.  So suppose that $R_{k+1}$ has index zero.  Consulting \reftab{DiskIndices} we deduce that $R_{k+1}$ is a boundary trigon. 
This proves hypothesis \refitm{Trigon}.  Note that hypothesis \refitm{Bigon} follows because $B_k$ meets $R_{k+1}$ along $b_k$. 

Let $s_{k+1} = \bdy R_{k+1} - (d_{k+1} \cup b^\circ_k)$ be the remaining side of the boundary trigon $R_{k+1}$.  By \reflem{Minimality}\refitm{NormalPlus} there is at most one switch in the interior of $s_{k+1}$. Let $c'_{k+1} = s_{k+1} - b^\circ_{k+1} - \bdy D$.  Note that $\gamma'_{k+1} = \gamma'_k \cup c'_{k+1}$.

The path $H(s_{k+1})$ has endpoints $H(p')$ and the lower endpoint of $H(d_{k+1})$. The point $H(p')$ lies on the edge $e_k$. Recall that $f_{k+1}$ is the face containing the lower endpoint of $H(d_{k+1})$. There are two cases, depending on the type of $H(d_{k+1})$. 
\begin{itemize}
\item
Suppose that $H(d_{k+1})$ has type $\textsc{b}_1$.  Then $\gamma_{k+1} = \gamma_k$.  In this case, $e_k$ is a boundary edge of $f_{k+1}$. Since there is at most one switch in the interior of $s_{k+1}$, there are in fact no such switches. So $s_{k+1} = b_{k+1}$ and $c'_{k+1}$ is a single switch, equal to $p'$. We deduce that $\gamma'_{k+1} = \gamma'_k$.   Since $\gamma_{k+1} = \gamma_k$ and  $\gamma'_{k+1} = \gamma'_k$, we set $F_{k+1} = F_k$. See \reffig{B1}.

\item
Suppose that $H(d_{k+1})$ has type $\textsc{b}_2$. Then $\gamma_{k+1} = \gamma_k \cup c_{k+1}$.
Let $t_{k+1}$ be the tetrahedron containing $H(R_{k+1})$. In this case, the path $H(s_{k+1})$ must cross the bottom edge of $t_{k+1}$ in order to get into $f_{k+1}$. Since there is at most one switch in the interior of $s_{k+1}$, there is exactly one. Let $f$ be the other lower face of $t_{k+1}$. Thus $c_{k+1}$ is a normal arc in $f$. Note that $H(c'_{k+1})$ is a properly immersed arc in $f$, with endpoints on the same edges as those of $c_{k+1}$. Thus there is a transverse homotopy $E$ taking $c_{k+1}$ to $H(c'_{k+1})$. Reparametrising $E$, we set $F_{k+1} = F_k \cup E$. See \reffig{B2}.
\end{itemize}
This proves hypothesis \refitm{Homotopy}.
\end{proof}

\begin{figure}[htbp]
\vspace{-10pt}
\centering
\subfloat[$H(d_{k+1})$ has type $\textsc{b}_1$.]{
\labellist
\small\hair 2pt
\pinlabel $\gamma$ at 12 38
\pinlabel $d_{k+1}$ at 91 98
\pinlabel $b_k$ at 64 72
\pinlabel $b_{k+1}$ at 93 72
\pinlabel $p'$ at 33 5
\pinlabel $p$ at 45 46
\pinlabel $\epsilon_k$ at 32 26
\endlabellist
\includegraphics[width=0.45\textwidth]{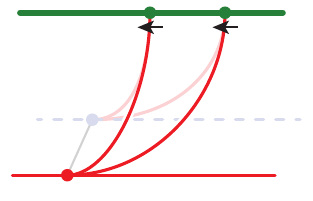}
\label{Fig:B1}
}
\quad
\subfloat[$H(d_{k+1})$ has type $\textsc{b}_2$.]{
\labellist
\small\hair 2pt
\pinlabel $\gamma$ at 12 38
\pinlabel $d_{k+1}$ at 91 98
\pinlabel $b_k$ at 64 72
\pinlabel $b_{k+1}$ at 95 72
\pinlabel $p'$ at 33 5
\pinlabel $p$ at 45 46
\pinlabel $c'_{k+1}$ at 53 4
\pinlabel $\epsilon_{k+1}$ at 73 25.6
\pinlabel $\epsilon_k$ at 32 26
\endlabellist
\includegraphics[width=0.45\textwidth]{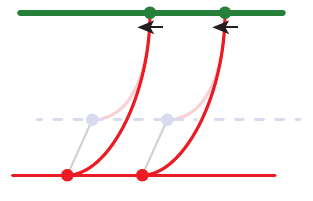}
\label{Fig:B2}
}
\caption{Extending the transverse homotopy $F_k$. As in \reffig{DownhillFromMaxBigon}, we have omitted applying $H$ to labels of subsets of $D$.}
\label{Fig:ExtendingTransverseHomotopy}
\end{figure}

Let $B_\triangleright = B_{K-1}$. This is the \emph{right-bigon} for $R_0$. We rerun the argument of \refclm{Induct} to the left to obtain the \emph{left-bigon} for $R_0$, denoted by $B_\triangleleft$. 

The induction in the proof of \refclm{Induct} extends to show that $R_K$ contains a cusp pointing at $\gamma'_{K-1}$. The cusp lies between $b_{K-1}$ and another branch on the boundary of $R_K$, which we call $c'_{K}$.  See the far right of \reffig{DownhillFromMaxBigon}. Let $Q_\triangleright$ be a small closed regular neighbourhood of $b_{K-1}$ in $R_K$. The boundary of the subregion $Q_\triangleright$ has four sides; we call it a \emph{right-quadrilateral}. The four sides are $d_K \cap N, b_{K-1}, c_K \cap N,$ and a fourth side, $s_\triangleright$ say. Note that $s_\triangleright$ is properly embedded in $R_K$. The quadrilateral $Q_\triangleright$ therefore has one cusp and three corners, and so it has index $-1/4$. 
An identical argument builds the \emph{left-quadrilateral} $Q_\triangleleft$.

Let $S(R_0) = Q_\triangleleft \cup B_\triangleleft \cup B_\triangleright \cup Q_\triangleright$.

\begin{claim}
\label{Clm:Rectangles}
For any max-bigons $R$ and $R'$
\begin{enumerate}
\item 
\label{Itm:Embedded}
$S(R)$ is embedded in $D$,
\item 
\label{Itm:Disjoint}
if $R\neq R'$ then $S(R)$ and $S(R')$ are disjoint, and 
\item 
\label{Itm:Rectangle}
$S(R)$ is a rectangle.  
\end{enumerate}
\end{claim}

\begin{proof}
Let $B_\triangleleft$ and $B_\triangleright$ be the right- and left-bigons for $R$; define $B'_\triangleleft$ and $B'_\triangleright$ similarly for $R'$. Note that boundary trigons in $B_\triangleleft$ have transverse orientations on their branches that disagree with the tangential orientation on $\bdy D$. On the other hand, boundary trigons in $B'_\triangleright$ have transverse orientations that agree with the tangential orientation. 

This proves that $B_\triangleleft$ and $B_\triangleright$ share only one region: the max-bigon itself, and so $B= B_\triangleleft \cup B_\triangleright$ is again a boundary bigon. The same argument shows that $B$ and $B' = B'_\triangleleft \cup B'_\triangleright$ have no regions in common if $R\neq R'$.

We claim that $\bdy B$ and $\bdy B'$ are disjoint. To see this, note that $\bdy B$ consists of an arc in $\bdy D$, and an arc in $\tau$. The transverse orientation on the arc in $\tau$ points into $B$, and similarly for $B'$. 

Let $Q_\triangleleft$, $Q_\triangleright$, $Q'_\triangleleft$ and $Q'_\triangleright$ be the quadrilaterals for $R$ and $R'$.  Since these are obtained by taking subsets of small regular neighbourhoods of branches in $\bdy B$ and $\bdy B'$, these are all pairwise disjoint (if $R\neq R'$). This proves parts \refitm{Embedded} and \refitm{Disjoint}.

Adding the subregions $Q_\triangleleft$ and $Q_\triangleright$ replaces the two corners of $B$ with four corners, and thus $S(R)$ is a rectangle, and we obtain \refitm{Rectangle}.
\end{proof}


Let $D' = D - \cup S(R)$, where the union ranges over all max-bigons $R$.
\begin{claim}
The induced cellulation of $D'$ has no regions of positive index.
\end{claim}

\begin{proof}
Suppose that $R'$ is a region of $D'$ having positive index.  
If $R'$ is a non-complete nullgon, or monogon, then $R'$ is also a region of $D - \tau$, contradicting \reflem{Minimality}. 
Also, $R'$ is not a boundary bigon since we removed them all.  
Thus $R'$ was created by cutting quadrilaterals out of some region $R$ of $D - \tau$.  
Note that $R'$ meets $\tau$, meets $\bdy D$ and meets $\bdy Q_\triangleright$ (say) along some side $s_\triangleright$. 
So $R'$ has at least three corners.  
Since its index is positive, $R'$ has exactly three corners.  
Thus $R = R' \cup Q_\triangleright$ is a boundary trigon, contradicting \refclm{NotTrigon}.
\end{proof}

Note that $D'$ has both outward and \emph{inward} corners (a combinatorial version of the exterior angle being $3\pi/2$).  Again following ~\cite[page~57]{Mosher03}, we generalise our definition of index; each inward corner adds $+1/4$ to the overall index.  Thus $D'$ has non-positive index.  Since rectangles have index zero, from the additivity of index we deduce that $D$ has non-positive index, a contradiction. This concludes the proof of \refthm{NormalLoopsEssential}. 
\end{proof}

\renewcommand{\UrlFont}{\tiny\ttfamily}
\renewcommand\hrefdefaultfont{\tiny\ttfamily}

\bibliographystyle{plainurl}
\bibliography{../bibfile.bib}
\end{document}

%% file: header_basic.tex
\usepackage{amsmath} 
\usepackage{amsthm} 
\usepackage{amssymb} 

\usepackage{microtype} 
\usepackage{pinlabel} 

\usepackage{MnSymbol} 

\usepackage[scaled=0.9]{sourcecodepro} 

\usepackage[hidelinks, pagebackref]{hyperref}

\makeatletter
\define@key{href}{font}{#1}
\makeatother
\usepackage{xpatch}
\newcommand\hrefdefaultfont{\ttfamily}
\xpatchcmd\href{\setkeys{href}{#1}}{\setkeys{href}{font=\hrefdefaultfont,#1}}{}{\fail}

\renewcommand*{\backref}[1]{}
\renewcommand*{\backrefalt}[4]{
  \ifcase #1 
  [No citations.]
  \or [#2]
  \else [#2]
  \fi }

\let\originalleft\left
\let\originalright\right
\renewcommand{\left}{\mathopen{}\mathclose\bgroup\originalleft}
\renewcommand{\right}{\aftergroup\egroup\originalright}

\newcommand{\calB}{\mathcal{B}}

\newcommand{\calF}{\mathcal{F}}

\newcommand{\calT}{\mathcal{T}}

\newcommand{\from}{\colon} 
 
\newcommand{\cover}[1]{{\widetilde{#1}}}

\newcommand{\corners}{\operatorname{corners}}

\newcommand{\bdy}{\partial} 



%% file: header_subtle.tex
\theoremstyle{plain}
\newtheorem{XXXtheoremQED}[equation]{Theorem} 
  {\pushQED{\qed}\begin{XXXtheoremQED}}
  {\popQED\end{XXXtheoremQED}}

\newcommand{\fakeenv}{} 

\newenvironment{restate}[2]  
{ 
 \renewcommand{\fakeenv}{#2} 
 \theoremstyle{plain} 
 \newtheorem*{\fakeenv}{#1~\ref{#2}} 
 \begin{\fakeenv}
}
{
 \end{\fakeenv}
}

\newenvironment{restated}[2]  
{ 
 \renewcommand{\fakeenv}{#2} 
 \theoremstyle{definition} 
 \newtheorem*{\fakeenv}{#1~\ref{#2}} 
 \begin{\fakeenv}
}
{
 \end{\fakeenv}
}

